\documentclass[leqno,12pt]{article} 
\usepackage{amsmath, amssymb}
\usepackage{amsthm} 
\theoremstyle{plain} 
\newtheorem{theorem}{\indent\sc Theorem}[section]
\newtheorem{lemma}[theorem]{\indent\sc Lemma}
\newtheorem{corollary}[theorem]{\indent\sc Corollary}
\newtheorem{proposition}[theorem]{\indent\sc Proposition}

\theoremstyle{definition} 
\newtheorem{definition}[theorem]{\indent\sc Definition}
\newtheorem{remark}[theorem]{\indent\sc Remark}

\title{Generalized quasi-statistical structures}
\author{Adara M. Blaga and Antonella Nannicini}

\date{}
\begin{document}

\maketitle

\markboth{{\small\it {\hspace{4cm} Generalized quasi-statistical structures}}}{\small\it{Generalized quasi-statistical structures
\hspace{4cm}}}

\footnote{ 
2010 \textit{Mathematics Subject Classification}.
53C15, 53B05, 53D05.
}
\footnote{ 
\textit{Key words and phrases}.
Quasi-statistical structures, generalized geometry, Patterson-Walker metric, Sasaki metric, Norden structures.
}

\begin{abstract}
Given a non-degenerate $(0,2)$-tensor field $h$ on a smooth manifold $M$, we consider a natural generalized complex and a generalized product structure on the generalized tangent bundle $TM\oplus T^*M$ of $M$ and we show that they are $\nabla$-integrable, for $\nabla$ an affine connection on $M$, if and only if $(M,h,\nabla)$ is a quasi-statistical manifold. We introduce the notion of generalized quasi-statistical structure and we prove that any quasi-statistical structure on $M$ induces  generalized quasi-statistical structures on $TM\oplus T^*M$. In this context,  dual connections are considered and some of their properties are established. The results are described in terms of  Patterson-Walker and Sasaki metrics on $T^*M$,  horizontal lift and Sasaki metrics on $TM$ and, when the connection $\nabla$ is flat, we define prolongation of quasi-statistical structures on manifolds to their cotangent and tangent bundles via generalized geometry. Moreover, Norden and Para-Norden structures are defined on $T^*M$ and $TM$.
\end{abstract}

\bigskip

\section{Introduction}

Statistical manifolds were introduced in \cite{a}, \cite{n}. They are manifolds of probability distributions, used in Information Geometry and related to Codazzi tensors and Affine Geometry. Let $h$ be a pseudo-Riemannian metric and let $\nabla$ be a torsion-free affine connection on a smooth manifold $M$. Then $(M,h,\nabla)$ is called a \textit{statistical manifold} if $({\nabla}_Xh)(Y,Z)=({\nabla}_Yh)(X,Z)$, for all $X,Y,Z\in C^{\infty}(TM)$. The definition can be extended to $(0,2)$-tensor fields and affine connections $\nabla$ with torsion, $T^{\nabla}$. In this case, $(h,\nabla)$ is called a \textit{quasi-statistical structure} on $M$ if $d^{\nabla}h=0$, where $ (d^{\nabla}h)(X,Y,Z):=({\nabla}_Xh)(Y,Z)-({\nabla}_Yh)(X,Z)+h(T^{\nabla}(X,Y),Z)$, for all $X,Y,Z\in C^{\infty}(TM)$, and the triple $(M,h,\nabla)$ is called a \textit{quasi-statistical manifold}.

In this paper, given a non-degenerate $(0,2)$-tensor field $h$ and an affine connection $\nabla$ on a smooth manifold $M$, we consider a natural generalized complex and a generalized product structure on the generalized tangent bundle $TM\oplus T^*M$ of $M$ and we show that they are $\nabla$-integrable if and only if $(M,h,\nabla)$ is a quasi-statistical manifold. We introduce the notion of \textit{generalized quasi-statistical structure} and we prove that any quasi-statistical structure on $M$, defined by a symmetric or skew-symmetric tensor, induces two natural generalized quasi-statistical structures on $TM\oplus T^*M$. We compute the dual connections and study some of their properties. The results are described in terms of  Patterson-Walker and Sasaki metrics on $T^*M$,  horizontal lift and Sasaki metrics on $TM$. In the case, the connection $\nabla$ is flat we can define prolongation of quasi-statistical structures on manifolds to their cotangent and tangent bundles via generalized geometry. Moreover, in the last section, we construct Norden and Para-Norden structures on $T^*M$ and $TM$.

\section{Quasi-statistical structures and generalized structures induced}

Let $M$ be a smooth manifold and $h$ a non-degenerate $(0,2)$-tensor field on $M$.
On the generalized tangent bundle $TM\oplus T^*M$ of $M$, we shall consider the generalized complex structure
\begin{equation} \label{m10}\hat{J}_c:=\begin{pmatrix}
                 0 & -h^{-1} \\
                 h & 0 \\
               \end{pmatrix}
\end{equation}
and the generalized product structure
\begin{equation} \label{m20}\hat{J}_p:=\begin{pmatrix}
                 0 & h^{-1} \\
                 h & 0 \\
               \end{pmatrix},
\end{equation}
where we denoted by $h$ the musical isomorphism, $\flat_h:TM\rightarrow T^*M$, $\flat_h(X):=i_Xh$, and by $h^{-1}$ its inverse, ${\sharp}_h:T^*M\rightarrow TM$.

Let
\begin{equation}\label{m1}
<X+\alpha,Y+\beta>:=-\frac{1}{2}(\alpha(Y)+\beta(X))
\end{equation}
be the natural indefinite metric on $TM\oplus T^*M$ and
\begin{equation}\label{m2}
(X+\alpha,Y+\beta):=-\frac{1}{2}(\alpha(Y)-\beta(X))
\end{equation}
be the natural symplectic structure on $TM\oplus T^*M$.

\begin{remark}
i) If $h$ is symmetric, then:
$$<\hat{J}_c\sigma,\hat{J}_c\tau>=-<\sigma,\tau> \ \
\textit{and} \ \ (\hat{J}_c\sigma,\hat{J}_c\tau)=(\sigma,\tau),$$$$<\hat{J}_p\sigma,\hat{J}_p\tau>=<\sigma,\tau> \ \
\textit{and} \ \ (\hat{J}_p\sigma,\hat{J}_p\tau)=-(\sigma,\tau),$$
or, equivalently:
$$<\hat{J}_c\sigma,\tau>=<\sigma,\hat{J}_c\tau> \ \
\textit{and} \ \ (\hat{J}_c\sigma,\tau)=-(\sigma,\hat{J}_c\tau),$$
$$<\hat{J}_p\sigma,\tau>=<\sigma,\hat{J}_p\tau> \ \
\textit{and} \ \ (\hat{J}_p\sigma,\tau)=-(\sigma,\hat{J}_p\tau),$$
for any $\sigma$, $\tau\in C^{\infty}(TM\oplus T^*M)$.

i) If $h$ is skew-symmetric, then:
$$<\hat{J}_c\sigma,\hat{J}_c\tau>=<\sigma,\tau> \ \
\textit{and} \ \ (\hat{J}_c\sigma,\hat{J}_c\tau)=-(\sigma,\tau),$$
$$<\hat{J}_p\sigma,\hat{J}_p\tau>=-<\sigma,\tau> \ \
\textit{and} \ \ (\hat{J}_p\sigma,\hat{J}_p\tau)=(\sigma,\tau),$$
or, equivalently:
$$<\hat{J}_c\sigma,\tau>=-<\sigma,\hat{J}_c\tau> \ \
\textit{and} \ \ (\hat{J}_c\sigma,\tau)=(\sigma,\hat{J}_c\tau),$$
$$<\hat{J}_p\sigma,\tau>=-<\sigma,\hat{J}_p\tau> \ \
\textit{and} \ \ (\hat{J}_p\sigma,\tau)=(\sigma,\hat{J}_p\tau),$$
for any $\sigma$, $\tau\in C^{\infty}(TM\oplus T^*M)$.
\end{remark}

On $TM \oplus T^*M$ we consider the bilinear form:
\begin{equation}\label{m3}
\check h(X+\alpha, Y+\beta):=h(X,Y)+h(h^{-1}(\alpha),h^{-1}(\beta)),
\end{equation}
for any $X,Y \in C^{\infty}(TM)$ and $\alpha, \beta \in C^{\infty}(T^*M).$

A direct computation gives the following:

\begin{lemma} The structures $\hat{J}_c$ and $\hat{J}_p$ satisfy respectively:
$$\check h(\hat{J}_c \sigma,\tau)=2(\sigma,\tau),$$
$$\check h(\sigma, \hat{J}_p \tau)=2<\sigma,\tau>,$$
for any $\sigma, \tau \in C^{\infty}(TM \oplus T^*M)$.
\end{lemma}

For $\nabla$ an affine connection on $M$, we consider the bracket $[\cdot,\cdot]_{\nabla}$ on $C^{\infty}(TM\oplus T^*M)$ \cite{na}:
$$[X+\alpha,Y+\beta]_{\nabla}:=[X,Y]+\nabla_X\beta-\nabla_Y\alpha,$$
for all $X,Y\in C^{\infty}(TM)$ and $\alpha,\beta\in C^{\infty}(T^*M)$.

A generalized complex or product structure $\hat{J}$ is called {\it $\nabla$-integrable} if its Nijenhuis tensor field $N_{\hat{J}}^{\nabla}$ with respect to $\nabla$:
$$N_{\hat{J}}^{\nabla}(\sigma, \tau):=[\hat{J}\sigma,\hat{J}\tau]_{\nabla}-\hat{J}[\hat{J}\sigma, \tau]_{\nabla}-\hat{J}[\sigma, \hat{J}\tau]_{\nabla} +\hat{J}^{2}[\sigma, \tau]_{\nabla}$$ vanishes for all $\sigma,\tau \in C^{\infty}(TM\oplus T^*M)$.

\bigskip

Let $M$ be a smooth manifold with a non-degenerate $(0,2)$-tensor field $h$ and an affine connection $\nabla$.

\begin{definition}\cite{ma}
We call $(h,\nabla)$ a \textit{quasi-statistical structure} (respectively, $(M,h,\nabla)$ a \textit{quasi-statistical manifold}) if $d^{\nabla}h=0$, where
$$(d^{\nabla}h)(X,Y,Z):=(\nabla_Xh)(Y,Z)-(\nabla_Yh)(X,Z)+h(T^{\nabla}(X,Y),Z),$$
for any $X,Y,Z\in C^{\infty}(TM)$ and $T^{\nabla}(X,Y):=\nabla_XY-\nabla_YX-[X,Y]$.
\end{definition}

We can state:
\begin{proposition}
The structures $\hat{J}_c$ and $\hat{J}_p$ are integrable if and only if $(M,h,\nabla)$ is a quasi-statistical manifold.
\end{proposition}
\begin{proof} In this proof we will shortly denote $\hat{J}_\mp$ for $\hat{J}_c=:\hat{J}_-$ and $\hat{J}_p=:\hat{J}_+$.

Let us compute:
$$N_{\hat{J}_\mp}^{\nabla}(X, Y)=[\hat{J}_\mp X,\hat{J}_\mp Y]_{\nabla}-\hat{J}_\mp [\hat{J}_\mp X,Y]_{\nabla}-\hat{J}_\mp[X, \hat{J}_\mp Y]_{\nabla} +\hat{J}^{2}_\mp[X, Y]_{\nabla}=$$
$$=\pm h^{-1}(({\nabla}_X h)Y-({\nabla}_Y h)X+h({\nabla}_X Y-{\nabla}_Y X-[X,Y]))=$$
$$=\pm h^{-1}((d^{\nabla} h)(X,Y))$$
$$N_{\hat{J}_\mp}^{\nabla}(X,h(Y))=[\hat{J}_\mp X,\hat{J}_\mp h(Y)]_{\nabla}-\hat{J}_\mp [\hat{J}_\mp X,h(Y)]_{\nabla}-\hat{J}_\mp[X, \hat{J}_\mp h(Y)]_{\nabla} +\hat{J}^{2}_\mp[X,h(Y)]_{\nabla}=$$
$$=\mp (({\nabla}_X h)Y-({\nabla}_Y h)X+h({\nabla}_X Y-{\nabla}_Y X-[X,Y]))=$$
$$=\mp (d^{\nabla} h)(X,Y)$$
$$N_{\hat{J}_\mp}^{\nabla}(h(X),h(Y))=[\hat{J}_\mp h(X),
\hat{J}_\mp h(Y)]_{\nabla}-\hat{J}_\mp [\hat{J}_\mp h(X),h(Y)]_{\nabla}-\hat{J}_\mp[h(X), \hat{J}_\mp h(Y)]_{\nabla} +$$
$$+\hat{J}^{2}_\mp [h(X),h(Y)]_{\nabla}=$$
$$=-h^{-1}(({\nabla}_X h)Y-({\nabla}_Y h)X+h({\nabla}_X Y-{\nabla}_Y X-[X,Y]))=$$
$$=-h^{-1}((d^{\nabla} h)(X,Y)),$$
for any $X,Y\in C^{\infty}(TM).$ Therefore the proof is complete.
\end{proof}

\section{Generalized quasi-statistical structures}

\begin{definition}
We call $D:C^{\infty}(TM\oplus T^*M)\times C^{\infty}(TM\oplus T^*M)\rightarrow C^{\infty}(TM\oplus T^*M)$ an \textit{affine connection} on $TM\oplus T^*M$ if it is $\mathbb{R}$-bilinear and for any $f\in C^{\infty}(M)$ and $\sigma,\tau \in C^{\infty}(TM\oplus T^*M)$, we have:
\begin{enumerate}
  \item $D_{f\sigma}\tau=f D_{\sigma}\tau$,
  \item $D_{\sigma}(f\tau)=\sigma(f)\tau+f D_{\sigma}\tau$,
\end{enumerate}
where $(X+\alpha)(f):=X(f)$, for $X+\alpha \in C^{\infty}(TM\oplus T^*M)$.
\end{definition}

Let $\hat{h}$ be a non-degenerate $(0,2)$-tensor field and $D$ an affine connection on the generalized tangent bundle $TM\oplus T^*M$ of the smooth manifold $M$.

\begin{definition}
We call $(\hat{h},D)$ a \textit{generalized quasi-statistical structure} if $d^{D}\hat{h}=0$, where
$$(d^{D}\hat{h})(\sigma,\tau,\nu):=(D_{\sigma}\hat{h})(\tau,\nu)-(D_{\tau}\hat{h})(\sigma,\nu)+
\hat{h}(T^{D}(\sigma,\tau),\nu),$$
for any $\sigma,\tau,\nu\in C^{\infty}(TM\oplus T^*M)$ and $T^{D}(\sigma,\tau):=D_{\sigma}\tau-D_{\tau}\sigma-[\sigma,\tau]_{\nabla}$, with $\nabla$ a given connection on $M$.
\end{definition}

\subsection{Generalized quasi-statistical structures induced by quasi-statistical structures}

Let $h$ be a non-degenerate $(0,2)$-tensor field and let $\nabla$ be an affine connection on $M$. We define the affine connection $\hat{\nabla}$ on $TM \oplus T^*M$ by:
\begin{equation}\label{m5}
\hat{\nabla}_{X+\alpha}Y+\beta:=\nabla_XY+h(\nabla_Xh^{-1}(\beta)),
\end{equation}
for any $X,Y \in C^{\infty}(TM)$ and $\alpha, \beta \in C^{\infty}(T^*M).$

\begin{theorem}
$(TM\oplus T^*M, \hat{h},\hat{\nabla})$ is a generalized quasi-statistical manifold if and only if $(M,h,\nabla)$ is a quasi-statistical manifold, where $\hat{h}$ is precisely $<\cdot,\cdot>$ or $(\cdot,\cdot)$ given by (\ref{m1}) and (\ref{m2}) respectively, according as $h$ is symmetric or skew-symmetric, and $\hat{\nabla}$ is given by (\ref{m5}).
\end{theorem}
\begin{proof}
First notice that the torsion of $\hat{\nabla}$ equals to
$$T^{\hat{\nabla}}(X+\alpha, Y+\beta):=\hat{\nabla}_{X+\alpha}Y+\beta-\hat{\nabla}_{Y+\beta}X+\alpha-[X+\alpha, Y+\beta]_{\nabla}=
$$$$=T^{\nabla}(X,Y)+h(\nabla_Xh^{-1}(\beta)-\nabla_Yh^{-1}(\alpha))-\nabla_X\beta+\nabla_Y\alpha.$$

We have:
$$(d^{\hat{\nabla}}\hat{h})(X+\alpha, Y+\beta,Z+\gamma):=(\hat{\nabla}_{X+\alpha}\hat{h})(Y+\beta,Z+\gamma)-(\hat{\nabla}_{Y+\beta}\hat{h})(X+\alpha,Z+\gamma)+$$$$
+\hat{h}(T^{\hat{\nabla}}(X+\alpha, Y+\beta),Z+\gamma):=$$
$$:=X(\hat{h}(Y+\beta,Z+\gamma))-\hat{h}(\hat{\nabla}_{X+\alpha}Y+\beta,Z+\gamma)-\hat{h}(Y+\beta,\hat{\nabla}_{X+\alpha},Z+\gamma)-$$
$$-Y(\hat{h}(X+\alpha,Z+\gamma))+\hat{h}(\hat{\nabla}_{Y+\beta}X+\alpha,Z+\gamma)+\hat{h}(X+\alpha,\hat{\nabla}_{Y+\beta},Z+\gamma)+$$
$$+\hat{h}(T^{\hat{\nabla}}(X+\alpha, Y+\beta),Z+\gamma):=$$
$$:=-\frac{1}{2}[X(\beta(Z)\pm \gamma(Y))-h(\nabla_Xh^{-1}(\beta),Z)\mp \gamma(\nabla_XY)-\beta(\nabla_XZ)\mp h(\nabla_Xh^{-1}(\gamma),Y)]+$$
$$+\frac{1}{2}[Y(\alpha(Z)\pm \gamma(X))-h(\nabla_Yh^{-1}(\alpha),Z)\mp \gamma(\nabla_YX)-\alpha(\nabla_YZ)\mp h(\nabla_Yh^{-1}(\gamma),X)]-$$
$$-\frac{1}{2}[h(\nabla_Xh^{-1}(\beta),Z)-h(\nabla_Yh^{-1}(\alpha),Z)-(\nabla_X\beta)Z+(\nabla_Y\alpha)Z]\mp\frac{1}{2}\gamma(T^{\nabla}(X,Y)):=$$
$$:=-\frac{1}{2}[\pm X(\gamma(Y))\mp \gamma(\nabla_XY)\mp h(\nabla_Xh^{-1}(\gamma),Y) \mp$$
$$\mp Y(\gamma(X))\pm \gamma(\nabla_YX)\pm h(\nabla_Yh^{-1}(\gamma),X)\pm \gamma(T^{\nabla}(X,Y)):=$$$$:=-\frac{1}{2}(d^{\nabla}h)(X,Y,h^{-1}(\gamma)).$$
Therefore the proof is complete.
\end{proof}

The couple $(\hat{h},\hat{\nabla})$ with $\hat{h}$ given by (\ref{m1}) or (\ref{m2}) respectively (according as $h$ is symmetric or skew-symmetric) and $\hat{\nabla}$ given by (\ref{m5}) will be called the \textit{generalized quasi-statistical structure induced} by $(h,\nabla)$.

A direct computation gives the expression of the \textit{generalized dual quasi-statistical connection} of $\hat{\nabla}$, precisely:

\begin{proposition}\label{p2}
Let $(M,h,\nabla)$ be a quasi-statistical manifold and let $(\hat{h},\hat{\nabla})$ be the generalized quasi-statistical structure induced on $TM\oplus T^*M$. Then the generalized dual quasi-statistical connection, ${\hat {\nabla}}^*$, defined by:
$$\hat h(Y+\beta, {\hat {\nabla}}^*_{X+\alpha}Z+\gamma)=X(\hat h(Y+\beta,Z+\gamma))-\hat h({\hat \nabla}_{X+\alpha}Y+\beta,Z+\gamma),$$
for all $X,Y,Z \in C^{\infty}(TM)$ and $\alpha, \beta, \gamma \in C^{\infty}(T^*M)$, is given by:
$${\hat {\nabla}}^*_{X+\alpha}Z+\gamma=h^{-1}({\nabla}_Xh(Z))+{\nabla}_X \gamma.$$
\end{proposition}
\begin{proof}
From the definition of the generalized dual quasi-statistical connection and using the definition of $\hat{\nabla}$, we get:
$$\hat{h}(Y+\beta,{\hat {\nabla}}^*_{X+\alpha}Z+\gamma)=X(\beta(Z))\pm X(\gamma(Y))-h(\nabla_Xh^{-1}(\beta),Z)\mp \gamma(\nabla_XY),$$
for any $X,Y,Z \in C^{\infty}(TM)$ and $\alpha, \beta, \gamma \in C^{\infty}(T^*M)$.

Let us denote ${\hat {\nabla}}^*_{X+\alpha}Z+\gamma=:V+\eta$. Then we have:
$$\beta(V)\pm \eta(Y)=X(\beta(Z))\pm X(\gamma(Y))-h(\nabla_Xh^{-1}(\beta),Z)\mp \gamma(\nabla_XY),$$
for any $X,Y,Z \in C^{\infty}(TM)$ and $\beta, \gamma \in C^{\infty}(T^*M)$.

Taking $\beta:=0$, we obtain:
$$\eta(Y)=X(\gamma(Y))-\gamma(\nabla_XY):=(\nabla_X\gamma)Y$$
and taking $Y:=0$, we obtain:
$$\beta(V)=X(\beta(Z))-h(\nabla_Xh^{-1}(\beta),Z)$$
which is equivalent to:
$$h(V,h^{-1}(\beta))=X(h(Z,h^{-1}(\beta)))- h(Z,\nabla_Xh^{-1}(\beta)):=(\nabla_Xh)(Z,h^{-1}(\beta))+h(\nabla_XZ,h^{-1}(\beta))$$
and to:
$$h(V)=(\nabla_Xh)(Z,\cdot)+h(\nabla_XZ)=\nabla_Xh(Z)$$
and to:
$$V=h^{-1}(\nabla_Xh(Z)).$$
Therefore the proof is complete.
\end{proof}

\begin{proposition}
Let $(M,h,\nabla)$ be a quasi-statistical manifold. Then ${\hat {\nabla}}^*$ is torsion-free.
\end{proposition}

\begin{proof} For all $X+\alpha, Y+\beta \in C^{\infty}(TM \oplus T^*M)$, we have:
$$T^{{\hat {\nabla}}^*}(X+\alpha,Y+\beta)={\hat {\nabla}}^*_{X+\alpha}Y+\beta - {\hat {\nabla}}^*_{Y+\beta}X+\alpha -[X+\alpha,Y+\beta]_{\nabla}=$$
$$=h^{-1}({\nabla}_X h(Y))- h^{-1}({\nabla}_Y h(X)) - [X,Y]=$$
$$=h^{-1}(({\nabla}_X h)Y- ({\nabla}_Y h)X +h(T^{\nabla}(X,Y)))=$$
$$=h^{-1}(d^{\nabla}(X,Y))=0.$$
\end{proof}

Let $h$ be a non-degenerate, symmetric or skew-symmetric $(0,2)$-tensor field on $M$ and let $\nabla$ be an affine connection on $M$. We have the following:

\begin{theorem}
$(TM\oplus T^*M, \check{h},\hat{\nabla})$ is a generalized quasi-statistical manifold if and only if $(M,h,\nabla)$ is a quasi-statistical manifold, where $\check h$ is given by (\ref{m3}) and $\hat{\nabla}$ is given by (\ref{m5}).
\end{theorem}

\begin{proof}
We have:
$$(d^{\hat{\nabla}}\check{h})(X+\alpha, Y+\beta,Z+\gamma):=(\hat{\nabla}_{X+\alpha}\check{h})(Y+\beta,Z+\gamma)-(\hat{\nabla}_{Y+\beta}\check{h})(X+\alpha,Z+\gamma)+$$$$
+\check{h}(T^{\hat{\nabla}}(X+\alpha, Y+\beta),Z+\gamma):=$$
$$:=X(\check{h}(Y+\beta,Z+\gamma))-\check{h}(\hat{\nabla}_{X+\alpha}Y+\beta,Z+\gamma)-\check{h}(Y+\beta,\hat{\nabla}_{X+\alpha},Z+\gamma)-$$
$$-Y(\check{h}(X+\alpha,Z+\gamma))+\check{h}(\hat{\nabla}_{Y+\beta}X+\alpha,Z+\gamma)+\check{h}(X+\alpha,\hat{\nabla}_{Y+\beta},Z+\gamma)+$$
$$+\check{h}(T^{\hat{\nabla}}(X+\alpha, Y+\beta),Z+\gamma):=$$
$$:=X(h(Y,Z))+X(\beta(h^{-1}(\gamma)))-\check h(\nabla_X Y +h(\nabla_X h^{-1}(\beta)), Z+\gamma)-$$
$$-\check h (Y+\beta,{\nabla}_X Z + h({\nabla}_X h^{-1}(\gamma)))-$$
$$-Y(h(X,Z))-Y(\alpha(h^{-1}(\gamma)))+\check h(\nabla_Y X +h(\nabla_Y h^{-1}(\alpha)), Z+\gamma)+$$
$$+\check h (X+\alpha,{\nabla}_Y Z + h({\nabla}_Y h^{-1}(\gamma)))+h(T^{\nabla}(X,Y),Z) \pm \gamma(({\nabla}_X h^{-1})\beta-({\nabla}_Y h^{-1})\alpha):=$$
$$:=(d^{\nabla}h)(X,Y,Z)+({\nabla}_X\beta)h^{-1}(\gamma)-({\nabla}_Y\alpha)h^{-1}(\gamma) \pm \gamma(h^{-1}({\nabla}_Y\alpha) - h^{-1}({\nabla}_X\beta)):=$$
$$:=(d^{\nabla}h)(X,Y,Z),$$
where the sign $+$ is for $h$ symmetric, $-$ is for $h$ skew-symmetric.
Therefore the proof is complete.
\end{proof}

\begin{proposition}\label{p3}
Let $(M,h,\nabla)$ be a quasi-statistical manifold and let $(\check{h},\hat{\nabla})$ be the generalized quasi-statistical structure induced on $TM\oplus T^*M$. Then the generalized dual quasi-statistical connection, ${{\hat {\nabla}}^*}_{\check h}$, defined by:
$$ \check h (Y+\beta, ({{{\hat{\nabla}}^*}}_{\check h})_{X+\alpha} Z+\gamma)=X(\check h(Y+\beta,Z+\gamma))-\check h({\hat \nabla}_{X+\alpha}Y+\beta,Z+\gamma),$$
for all $X,Y,Z \in C^{\infty}(TM)$ and $\alpha, \beta, \gamma \in C^{\infty}(T^*M)$, is given by:
$$({{{\hat{\nabla}}^*}}_{\check h})_{X+\alpha}Z+\gamma=h^{-1}({\nabla}_Xh(Z))+{\nabla}_X \gamma.$$
Therefore:
$${{{\hat{\nabla}}^*}}_{\check h}={\hat{\nabla}}^*.$$
\end{proposition}

\begin{proof}
We get:
$$\check{h}(Y+\beta,({{{\hat{\nabla}}^*}}_{\check h})_{X+\alpha}Z+\gamma)=X(h(Y,Z))+X(\beta(h^{-1}(\gamma)))-h(\nabla_X Y,Z) \mp \gamma(\nabla_Xh^{-1}(\beta))=$$
$$=X(h(Y,Z))-h(\nabla_X Y,Z)\mp \gamma(\nabla_Xh^{-1}(\beta))\pm X(\gamma(h^{-1}(\beta))),$$
for any $X,Y,Z \in C^{\infty}(TM)$ and $\alpha, \beta, \gamma \in C^{\infty}(T^*M)$.

Let us denote $({{{\hat{\nabla}}^*}}_{\check h})_{X+\alpha}Z+\gamma=:V+\eta$. Then we have:
$$h(Y,V)\pm \eta( h^{-1}(\beta))=X(h(Y,Z))-h({\nabla}_X Y,Z)\pm({\nabla}_X\gamma)h^{-1}(\beta),$$
for any $X,Y,Z \in C^{\infty}(TM)$ and $\beta \in C^{\infty}(T^*M)$.

Taking $Y:=0$, we obtain:
$$\eta(h^{-1}(\beta))=\nabla_Xh^{-1}(\beta)$$
and taking $\beta:=0$, we obtain:
$$h(Y,V)=(\nabla_Xh)(Y,Z)+h(Y,{\nabla}_X Z)$$
which is equivalent to:
$$h(V)=(\nabla_Xh)(Z,\cdot)+h(\nabla_XZ)=\nabla_Xh(Z)$$
and to:
$$V=h^{-1}({\nabla}_Xh(Z)).$$
Therefore the proof is complete.
\end{proof}

Given an affine connection $D$ on $TM \oplus T^*M$, we define the curvature operator of $D$, $R^D:C^{\infty}(TM \oplus T^*M)\times C^{\infty}(TM \oplus T^*M) \times C^{\infty}(TM \oplus T^*M) \rightarrow C^{\infty}(TM \oplus T^*M)$, on $\sigma, \tau, \nu \in C^{\infty}(TM \oplus T^*M)$, as in the following:
$$R^D(\sigma,\tau)\nu=(D_{\sigma}D_{\tau}-D_{\tau}D_{\sigma}-D_{[\sigma,\tau]_\nabla})\nu,$$
where $\nabla$ is a given connection on $M$.

\begin{proposition} Let $(M,h,\nabla)$ be a quasi-statistical manifold and let $(\hat h,\hat {\nabla})$ be the generalized quasi-statistical structure induced on $TM \oplus T^*M$. Then the curvature operators of $\hat {\nabla}$ and ${\hat {\nabla}}^*$ are given respectively by:
$$R^{\hat {\nabla}}(X+\alpha,Y+\beta)Z+\gamma=R^{\nabla}(X,Y)Z+h(R^{\nabla}(X,Y)h^{-1}(\gamma))$$
$$R^{{\hat {\nabla}}^*}(X+\alpha,Y+\beta)Z+\gamma=h^{-1}(R^{\nabla}(X,Y)h(Z))+R^{\nabla}(X,Y)\gamma,$$
where $X,Y,Z \in C^{\infty}(TM)$, $\alpha, \beta, \gamma \in  C^{\infty}(T^*M)$ and $R^{\nabla}$ is the curvature operator of $\nabla$.
In particular, $\hat {\nabla}$ and its dual ${\hat {\nabla}}^*$ are flat if and only if $\nabla$ is flat.
\end{proposition}

\begin{proof}
Let us compute:
$${{{\hat{\nabla}}}}_{X+\alpha} {{{\hat{\nabla}}}}_{Y+\beta} Z+\gamma - {{{\hat{\nabla}}}}_{Y+\beta}{{{\hat{\nabla}}}}_{X+\alpha}Z+\gamma - {{{\hat{\nabla}}}}_{[X+\alpha,Y+\beta]_{\nabla}}Z+\gamma:=$$
$$:={{{\hat{\nabla}}}}_{X+\alpha}({\nabla}_Y Z+h(\nabla_Yh^{-1}(\gamma)))-{{{\hat{\nabla}}}}_{Y+\beta}({\nabla}_X Z+h(\nabla_Xh^{-1}(\gamma)))-$$
$$-{\nabla}_{[X,Y]}Z-h(\nabla_{[X,Y]}h^{-1}(\gamma)):=$$
$$:=\nabla_X{\nabla}_Y Z+h({\nabla}_X \nabla_Yh^{-1}(\gamma)) -\nabla_Y{\nabla}_X Z- h({\nabla}_Y \nabla_Xh^{-1}(\gamma))- $$
$$-{\nabla}_{[X,Y]}Z-h(\nabla_{[X,Y]}h^{-1}(\gamma)):=$$
$$:= R^{\nabla}(X,Y)Z+h (R^{\nabla} (X,Y )h^{-1}(\gamma))$$
and:
$${{{\hat{\nabla}}^*}}_{X+\alpha} {{{\hat{\nabla}}^*}}_{Y+\beta} Z+\gamma -  {{{\hat{\nabla}}^*}}_{Y+\beta}{{{\hat{\nabla}}^*}}_{X+\alpha}Z+\gamma - {{{\hat{\nabla}}^*}}_{[X+\alpha,Y+\beta]_{\nabla}}Z+\gamma:=$$
$$:={{{\hat{\nabla}}^*}}_{X+\alpha}(h^{-1}(\nabla_Yh(Z))+{\nabla}_Y \gamma)-{{{\hat{\nabla}}^*}}_{Y+\beta}(h^{-1}(\nabla_Xh(Z))+{\nabla}_X \gamma)-$$
$$-h^{-1}(\nabla_{[X,Y]}h(Z))-{\nabla}_{[X,Y]}\gamma:=$$
$$:=h^{-1}({\nabla}_X \nabla_Yh(Z))+\nabla_X{\nabla}_Y \gamma - h^{-1}({\nabla}_Y \nabla_Xh(Z))-\nabla_Y{\nabla}_X \gamma- $$
$$-h^{-1}(\nabla_{[X,Y]}h(Z))-{\nabla}_{[X,Y]}\gamma:=$$
$$:= h^{-1} (R^{\nabla} (X,Y )h(Z))+ R^{\nabla}(X,Y) \gamma.$$
Therefore the proof is complete.
\end{proof}

\begin{proposition}
The structures $\hat{J}_c$ and $\hat{J}_p$ are $\hat{\nabla}$-parallel and ${\hat{\nabla}}^*$-parallel.
\end{proposition}

\begin{proof} In this proof we will shortly denote $\hat{J}_\mp$ for $\hat{J}_c=:\hat{J}_-$ and $\hat{J}_p=:\hat{J}_+$.
Let us compute:
$$(\hat{\nabla}_{X+\alpha}\hat{J}_\mp)Y+\beta:=\hat{\nabla}_{X+\alpha}(\mp h^{-1}(\beta)+h(Y))-\hat{J}_\mp(\hat{\nabla}_{X+\alpha}Y+\beta):=
$$$$:=\mp\nabla_Xh^{-1}(\beta)+h(\nabla_Xh^{-1}(h(Y)))-\hat{J}_\mp(\nabla_XY+h(\nabla_Xh^{-1}(\beta))):=$$
$$:=\mp \nabla_Xh^{-1}(\beta)+h(\nabla_XY)\pm h^{-1}(h(\nabla_Xh^{-1}(\beta)))-h(\nabla_XY)=0;$$
moreover:
$$({\hat{\nabla}}^*_{X+\alpha}\hat{J}_\mp)Y+\beta:={\hat{\nabla}}^*_{X+\alpha}(\mp h^{-1}(\beta)+h(Y))-\hat{J}_\mp({\hat{\nabla}}^*_{X+\alpha}Y+\beta):=
$$
$$:=\mp h^{-1}(\nabla_X \beta)+\nabla_X h(Y)-\hat{J}_\mp (h^{-1} (\nabla_Xh(Y))+\nabla_X \beta):=$$
$$:=\mp h^{-1}(\nabla_X \beta)+\nabla_X h(Y)-\nabla_X h(Y) \pm h^{-1}(\nabla_X \beta)=0,$$
for any $X,Y \in C^{\infty}(TM)$ and $\alpha, \beta \in C^{\infty}(T^*M).$
Therefore the proof is complete.
\end{proof}

\subsection{Generalized quasi-statistical structures induced by torsion-free connections}

Another affine connection on the generalized tangent bundle $TM\oplus T^*M$ is naturally defined by an affine connection $\nabla$ on $M$ by:
\begin{equation}\label{m4}
\check{\nabla}_{X+\alpha}Y+\beta:=\nabla_XY+\nabla_X\beta,
\end{equation}
for any $X,Y \in C^{\infty}(TM)$ and $\alpha, \beta \in C^{\infty}(T^*M).$

\begin{remark}\label{r}
One can check that if $h$ is a non-degenerate $(0,2)$-tensor field on $M$ which is  $\nabla$-parallel, then the connections $\hat{\nabla}$ and $\check{\nabla}$ coincide (since we have $\hat{\nabla}_{X+\alpha}Y+\beta-\check{\nabla}_{X+\alpha}Y+\beta=-(\nabla_Xh)(h^{-1}(\beta),\cdot)$, for any $X,Y \in C^{\infty}(TM)$ and $\alpha, \beta \in C^{\infty}(T^*M)$). In particular, $\hat{\nabla}^*=\hat \nabla=\check \nabla$.
\end{remark}

We have the following:

\begin{proposition}\label{p1}
$(TM\oplus T^*M, \hat{h},\check{\nabla})$ is a generalized quasi-statistical manifold if and only if $\nabla$ is torsion-free, where $\hat{h}$ is precisely $<\cdot,\cdot>$ or $(\cdot,\cdot)$ given by (\ref{m1}) and (\ref{m2}) respectively and $\check{\nabla}$ is given by (\ref{m4}).
\end{proposition}

\begin{proof}
First notice that the torsion of $\check{\nabla}$ equals to
$$T^{\check{\nabla}}(X+\alpha, Y+\beta):=\check{\nabla}_{X+\alpha}Y+\beta-\check{\nabla}_{Y+\beta}X+\alpha-[X+\alpha, Y+\beta]_{\nabla}=T^{\nabla}(X,Y).$$

We have:
$$(d^{\check{\nabla}}\hat{h})(X+\alpha, Y+\beta,Z+\gamma):=(\check{\nabla}_{X+\alpha}\hat{h})(Y+\beta,Z+\gamma)-(\check{\nabla}_{Y+\beta}\hat{h})(X+\alpha,Z+\gamma)+$$$$
+\hat{h}(T^{\check{\nabla}}(X+\alpha, Y+\beta),Z+\gamma):=$$
$$:=X(\hat{h}(Y+\beta,Z+\gamma))-\hat{h}(\check{\nabla}_{X+\alpha}Y+\beta,Z+\gamma)-\hat{h}(Y+\beta,\check{\nabla}_{X+\alpha},Z+\gamma)-$$
$$-Y(\hat{h}(X+\alpha,Z+\gamma))+\hat{h}(\check{\nabla}_{Y+\beta}X+\alpha,Z+\gamma)+\hat{h}(X+\alpha,\check{\nabla}_{Y+\beta},Z+\gamma)+$$
$$+\hat{h}(T^{\check{\nabla}}(X+\alpha, Y+\beta),Z+\gamma):=$$
$$:=-\frac{1}{2}[X(\beta(Z)\pm \gamma(Y))-(\nabla_X\beta)Z\mp \gamma(\nabla_XY)-\beta(\nabla_XZ)\mp (\nabla_X\gamma)Y]+$$
$$+\frac{1}{2}[Y(\alpha(Z)\pm \gamma(X))-(\nabla_Y\alpha)Z\mp \gamma(\nabla_YX)-\alpha(\nabla_YZ)\mp (\nabla_Y\gamma)X]\mp\frac{1}{2}\gamma(T^{\nabla}(X,Y)):=$$
$$:=-\frac{1}{2}[X(\beta(Z))\pm X(\gamma(Y))-X(\beta(Z))+\beta(\nabla_XZ) \mp\gamma(\nabla_XY)-\beta(\nabla_XZ)\mp X(\gamma(Y))\pm \gamma(\nabla_XY)]+$$
$$+\frac{1}{2}[Y(\alpha(Z))\pm Y(\gamma(X))-Y(\alpha(Z))+\alpha(\nabla_YZ)\mp \gamma(\nabla_YX)-\alpha(\nabla_YZ)\mp Y(\gamma(X))\pm \gamma(\nabla_YX)]\mp$$
$$\mp\frac{1}{2}\gamma(T^{\nabla}(X,Y))=\mp\frac{1}{2}\gamma(T^{\nabla}(X,Y)).$$
Therefore the proof is complete.
\end{proof}

\begin{proposition}\label{p4}
Let $\nabla$ be a torsion-free affine connection on $M$ and let 
\linebreak 
$(<\cdot,\cdot>, \check{\nabla})$ and $((\cdot,\cdot),\check{\nabla})$ be the canonical generalized quasi-statistical structures defined in Proposition \ref{p1}. Then $\check{\nabla}$ and its generalized dual quasi-statistical connection, ${\check{\nabla}}^*$, coincide.
\end{proposition}

\begin{proof}
Let us denote ${\check {\nabla}}^*_{X+\alpha}Z+\gamma=:V+\eta$.
From the definition of the generalized dual quasi-statistical connection and using the definition of $\check{\nabla}$, we get:
$$\beta(V)\pm \eta(Y)=X(\beta(Z))\pm X(\gamma(Y))-(\nabla_X\beta)Z\mp \gamma(\nabla_XY),$$
for any $X,Y,Z \in C^{\infty}(TM)$ and $\beta, \gamma \in C^{\infty}(T^*M)$.

Taking $\beta:=0$, we obtain:
$$\pm \eta(Y)=\pm X(\gamma(Y))\mp \gamma(\nabla_XY):=\pm (\nabla_X\gamma)Y$$
and taking $Y:=0$, we obtain:
$$\beta(V)=X(\beta(Z))-(\nabla_X\beta)Z:=\beta(\nabla_XZ).$$
Therefore the proof is complete.
\end{proof}

\begin{proposition}
If $\nabla$ is a torsion-free affine connection and $h$ is a $\nabla$-parallel $(0,2)$-tensor field on $M$, then $(\check{h},\check{\nabla})$ is a generalized quasi-statistical structure, where $\check{h}$ is given by (\ref{m3}) and $\check{\nabla}$ is given by (\ref{m4}).
\end{proposition}

\begin{proof}
We have:
$$(d^{\check{\nabla}}\check{h})(X+\alpha, Y+\beta,Z+\gamma):=(\check{\nabla}_{X+\alpha}\check{h})(Y+\beta,Z+\gamma)-(\check{\nabla}_{Y+\beta}\check{h})(X+\alpha,Z+\gamma)+$$
$$+\check{h}(T^{\check{\nabla}}(X+\alpha, Y+\beta),Z+\gamma):=$$
$$:=X(\check{h}(Y+\beta,Z+\gamma))-\check{h}(\check{\nabla}_{X+\alpha}Y+\beta,Z+\gamma)-\check{h}(Y+\beta,\check{\nabla}_{X+\alpha},Z+\gamma)-$$
$$-Y(\check{h}(X+\alpha,Z+\gamma))+\check{h}(\check{\nabla}_{Y+\beta}X+\alpha,Z+\gamma)+\check{h}(X+\alpha,\check{\nabla}_{Y+\beta},Z+\gamma)+$$
$$+\check{h}(T^{\check{\nabla}}(X+\alpha, Y+\beta),Z+\gamma):=$$
$$:=X(h(Y,Z)+h(h^{-1}(\beta),h^{-1}(\gamma)))-$$
$$-h(\nabla_XY,Z)-h(h^{-1}(\nabla_X\beta),h^{-1}(\gamma))-
h(Y,\nabla_XZ)-h(h^{-1}(\beta),h^{-1}(\nabla_X\gamma))-$$
$$-Y(h(X,Z)+h(h^{-1}(\alpha),h^{-1}(\gamma)))+$$
$$+h(\nabla_YX,Z)+h(h^{-1}(\nabla_Y\alpha),h^{-1}(\gamma))+
h(X,\nabla_YZ)+h(h^{-1}(\alpha),h^{-1}(\nabla_Y\gamma))+h(T^{\nabla}(X,Y),Z):=$$
$$:=(\nabla_Xh)(Y,Z)-(\nabla_Yh)(X,Z)+h(T^{\nabla}(X,Y),Z)+$$
$$+X(\beta(h^{-1}(\gamma)))-(\nabla_X\beta)h^{-1}(\gamma)-\beta(h^{-1}(\nabla_X\gamma))-$$
$$-Y(\alpha(h^{-1}(\gamma)))+(\nabla_Y\alpha)h^{-1}(\gamma)+\alpha(h^{-1}(\nabla_Y\gamma))=$$
$$=(\nabla_Xh)(Y,Z)-(\nabla_Yh)(X,Z)+h(T^{\nabla}(X,Y),Z)+$$
$$+\beta(\nabla_Xh^{-1}(\gamma))-\beta(h^{-1}(\nabla_X\gamma))-\alpha(\nabla_Yh^{-1}(\gamma))+\alpha(h^{-1}(\nabla_Y\gamma)).$$

Also, for any $V\in C^{\infty}(TM)$, we have:
$$h(h^{-1}(\nabla_X\gamma)-\nabla_Xh^{-1}(\gamma),V)=h(h^{-1}(\nabla_X\gamma),V)-h(\nabla_Xh^{-1}(\gamma),V)=$$
$$=(\nabla_X\gamma)V-h(\nabla_Xh^{-1}(\gamma),V):=X(\gamma(V))-\gamma(\nabla_XV)-h(\nabla_Xh^{-1}(\gamma),V)=$$
$$=X(h(h^{-1}(\gamma),V))-h(h^{-1}(\gamma),\nabla_XV)-h(\nabla_Xh^{-1}(\gamma),V):=(\nabla_Xh)(h^{-1}(\gamma),V)=0,$$
hence $h^{-1}(\nabla_X\gamma)-\nabla_Xh^{-1}(\gamma)=0$, for any $X\in C^{\infty}(TM)$ and $\gamma\in C^{\infty}(T^*M)$. Therefore, $d^{\check{\nabla}}\check{h}=0$ and the proof is complete.
\end{proof}

\begin{proposition}\label{p5}
Let $(M,h,\nabla)$ be a quasi-statistical manifold with $\nabla$ a torsion-free affine connection, $h$ a $\nabla$-parallel $(0,2)$-tensor field on $M$ and let $(\check{h},\check{\nabla})$ be the generalized quasi-statistical structure on $TM\oplus T^*M$, with $\check{h}$ given by (\ref{m3}) and $\check{\nabla}$ given by (\ref{m4}). Then $\check{\nabla}$ and its generalized dual quasi-statistical connection, ${{\check {\nabla}}^*}_{\check h}$, coincide.
\end{proposition}

\begin{proof}
We get:
$$\check{h}(Y+\beta,({{{\check{\nabla}}^*}}_{\check h})_{X+\alpha}Z+\gamma)=X(h(Y,Z))+X(h(h^{-1}(\beta),h^{-1}(\gamma)))-$$$$-h(\nabla_XY,Z)-h(h^{-1}(\nabla_X\beta), h^{-1}(\gamma))
=$$
$$=h(Y,\nabla_XZ)+\beta(\nabla_Xh^{-1}(\gamma)),$$
for any $X,Y,Z \in C^{\infty}(TM)$ and $\alpha, \beta, \gamma \in C^{\infty}(T^*M)$.

Let us denote $({{{\check{\nabla}}^*}}_{\check h})_{X+\alpha}Z+\gamma=:V+\eta$. Then we have:
$$h(Y,V)+h(h^{-1}(\beta),h^{-1}(\eta))=h(Y,\nabla_XZ)+\beta(\nabla_Xh^{-1}(\gamma)),$$
for any $X,Y,Z \in C^{\infty}(TM)$ and $\beta,\gamma \in C^{\infty}(T^*M)$.

Taking $Y:=0$, we obtain:
$$\beta(h^{-1}(\eta))=\beta(\nabla_Xh^{-1}(\gamma))$$
which is equivalent to:
$$\eta=h(\nabla_Xh^{-1}(\gamma))$$
and taking $\beta:=0$, we obtain:
$$h(Y,V)=h(Y,{\nabla}_X Z)$$
which is equivalent to:
$$V=\nabla_XZ.$$
From Remark \ref{r} we get ${{{\check{\nabla}}^*}}_{\check h}=\hat{\nabla}=\check{\nabla}$.
Therefore the proof is complete.
\end{proof}

\section{The pull-back tensors on $TM \oplus T^{*}M$ of horizontal lifts, Sasaki and Patterson-Walker metrics}

\subsection{Patterson-Walker and Sasaki metrics on $T^*M$}

Let $M$ be a smooth manifold and let $\nabla$ be an affine connection on $M$.

Let $\pi :T^*M \rightarrow M$ be the canonical projection and ${\pi}_*:T(T^*M) \rightarrow TM$ be the tangent map of $\pi$. If $a \in T^*M$ and $A \in T_a(T^*M)$, then ${\pi}_*(A) \in T_{\pi (a)}M$ and we denote by ${\chi}_a$ the standard identification between $T^*_{\pi (a)}M$ and its tangent space $ T_a(T^*_{\pi (a)}M)$.

Let ${\Phi}^{\nabla} :TM \oplus T^{*}M \rightarrow T(T^*M)$ be the bundle morphism defined by \cite{an}:
\begin{equation} \label{m30} {\Phi}^{\nabla}(X+\alpha):=X^H_a+{\chi}_a(\alpha),\end{equation}
where $a \in T^*M$ and $X^H_a$ is the horizontal lift of $X\in T_{\pi (a)}M$.

Let $\left\{ x^{1},...,x^{n}\right\} $ be local coordinates on $M$, let $\left\{ {\tilde{x}}^{1},..., {\tilde{x}}^{n},y_1,...,y_n\right\} $ be respectively the corresponding local coordinates on $T^*M$ and let $\{X_1,...,X_n, \dfrac{\partial }{\partial
{y_{1}}},.., \dfrac{\partial }{\partial
{y_{n}}}\}$ be a local frame on $T(T^*M)$, where $X_i=\dfrac{\partial }{\partial
{\tilde{x}}^{i}}$. The horizontal lift of $\dfrac{\partial }{\partial
{{x}}^{i}}$ is defined by:
$$(\dfrac{\partial }{\partial
{{x}}^{i}})^H:=X_i+y_k{\Gamma}^k_{il}{\dfrac{\partial }{\partial
y_{l}}}$$
and we will denote $X_i^H=:(\dfrac{\partial }{\partial
{{x}}^{i}})^H$. Moreover, the vertical lift of $\dfrac{\partial }{\partial
{{x}}^{i}}$ is defined by:
$$(\dfrac{\partial }{\partial
{{x}}^{i}})^V:=\dfrac{\partial }{\partial
{{y}}_{i}},$$
where $i,k,l$ run from $1$ to $n$ and $\Gamma^k_{il}$ are the Christoffel's symbols of $\nabla$.

Let ${\Phi}^{\nabla} :TM \oplus T^{*}M \rightarrow T(T^*M)$ be the bundle morphism defined before (which is an isomorphism on the fibres). In local coordinates, we have the following expressions:
$${\Phi}^{\nabla}\left(\dfrac{\partial }{\partial
{x^{i}}}\right)=X_i^H$$
$${\Phi}^{\nabla}\left({dx^j}\right)=\dfrac{\partial }{\partial
{y_{j}}}.$$

In \cite{pw}, starting from a torsion-free affine connection on $M$, the Patterson-Walker metric, $\tilde h$, on $T^*M$ is defined as in the following:
$$\tilde h(X^H,Y^H)=0$$
$$\tilde h(X^V,Y^V)=0$$
$$\tilde h(Y^V,X^H)=\tilde h(X^H,Y^V)=(({\Phi}^{\nabla})^{-1}(Y^V))(X),$$
where $X,Y \in C^{\infty}(T^*M)$, $X^H, Y^H$ are the horizontal lifts and $X^V,Y^V$ are the vertical lifts of $X,Y$ respectively.

The definition can also be given if $\nabla$ has torsion and we define ${\tilde h}_\pm$ on $T^*M$ as in the following:
$${\tilde h}_\pm (X^H,Y^H)=0$$
$${\tilde h}_\pm (X^V,Y^V)=0$$
$${\tilde h}_\pm (Y^V,X^H)=(({\Phi}^{\nabla})^{-1}(Y^V))(X)$$
$${\tilde h}_\pm (X^H,Y^V)=\pm (({\Phi}^{\nabla})^{-1}(Y^V))(X),$$
where $X,Y \in C^{\infty}(T^*M)$, $X^H, Y^H$ are the horizontal lifts and $X^V,Y^V$ are the vertical lifts of $X,Y$ respectively.

We denote by ${\tilde {\tilde h}}_\pm$ the pull-back tensors of ${\tilde h}_\pm$ on $TM \oplus T^*M$:
$${\tilde {\tilde h}}_\pm (\sigma,\tau):=({{\Phi}^{\nabla}})^*({\tilde h}_\pm)(\sigma,\tau):={\tilde  h}_\pm ({\Phi}^{\nabla}(\sigma),{\Phi}^{\nabla}(\tau)),$$
for any $\sigma, \tau \in C^{\infty}(TM \oplus T^*M)$. Remark that
${\tilde {\tilde h}}_\pm$ are related to the indefinite metric or to the symplectic structure of $TM \oplus T^*M$ as follows.

\begin{proposition} $${\tilde {\tilde h}}_+=-2<\cdot, \cdot >$$
$${\tilde {\tilde h}}_-=-2(\cdot , \cdot ).$$
\end{proposition}
\begin{proof} Let $\sigma=X+\alpha$, $\tau=Y+\beta$, $X,Y \in C^{\infty}(TM)$, $\alpha, \beta \in C^{\infty}(T^*M)$. Then:
$${\tilde {\tilde h}}_\pm(\sigma,\tau)={\tilde h}_\pm(X^H+{\Phi}^{\nabla}(\alpha),Y^H+{\Phi}^{\nabla}(\beta))=$$
$$={\tilde h}_\pm({\Phi}^{\nabla}(\alpha),Y^H)+{\tilde h}_\pm(X^H,{\Phi}^{\nabla}(\beta))=$$
$$=\alpha (Y) \pm \beta (X).$$
Then we get the statement.
\end{proof}

Let $h$ be a non-degenerate $(0,2)$-tensor field on $M$.
The Sasaki $(0,2)$-tensor field $h^{S^*}$ on $T^*M$, with respect to $\nabla$, is naturally defined by:
$$h^{S^*}(X^H,Y^H)=h(X,Y)$$
$$h^{S^*}(\alpha^V,\beta^V)=h(h^{-1}(\alpha),h^{-1}(\beta))$$
$$h^{S^*}(\alpha^V,Y^H)=0,$$
where $X,Y \in C^{\infty}(TM)$, $\alpha,\beta\in C^{\infty}(T^*M)$, $X^H, Y^H$ are the horizontal lifts of $X,Y$ and $\alpha^V,\beta^V$ are the vertical lifts of $\alpha,\beta$ respectively.

We denote by $\tilde{h}^{S^*}$ the pull-back tensor of $h^{S^*}$ on $TM \oplus T^*M$:
$$\tilde{h}^{S^*} (\sigma,\tau):=({{\Phi}^{\nabla}})^*(h^{S^*})(\sigma,\tau):= h^{S^*} ({\Phi}^{\nabla}(\sigma),{\Phi}^{\nabla}(\tau)),$$
for any $\sigma, \tau \in C^{\infty}(TM \oplus T^*M).$

\begin{proposition}
If $h$ is a non-degenerate $(0,2)$-tensor field on $M$, then:
$$\tilde{h}^{S^*} (X+\alpha,Y+\beta)=\hat h (X+\alpha,Y+\beta)=h(X,Y)+h(h^{-1}(\alpha),h^{-1}(\beta)),$$
for any $X,Y \in C^{\infty}(TM)$ and $\alpha, \beta \in C^{\infty}(T^*M).$
\end{proposition}
\begin{proof} Let $\sigma=X+\alpha$, $\tau=Y+\beta$, $X,Y \in C^{\infty}(TM)$, $\alpha, \beta \in C^{\infty}(T^*M)$. Then:
$$\tilde{h}^{S^*}(\sigma,\tau)=h^{S^*}(X^H+{\Phi}^{\nabla}(\alpha),Y^H+{\Phi}^{\nabla}(\beta))=$$
$$=h^{S^*}(X^H,Y^H)+h^{S^*}({\Phi}^{\nabla}(\alpha),{\Phi}^{\nabla}(\beta))=$$
$$=h(X,Y)+h^{S^*}({\Phi}^{\nabla}(\alpha),{\Phi}^{\nabla}(\beta)).$$

In local coordinates, let $\alpha={\alpha}_k dx^k$, $\beta={\beta}_ldx^l$ and we get:
$$h^{S^*}({\Phi}^{\nabla}(\alpha),{\Phi}^{\nabla}(\beta))=h^{S^*}({\alpha}_k \dfrac{\partial }{\partial
{y_{k}}},{\beta}_l\dfrac{\partial }{\partial
{y_{l}}})=$$
$$={\alpha}_k{\beta}_l h_{kl}=h(h^{-1}(\alpha),h^{-1}(\beta)).$$
Then we get the statement.
\end{proof}

\subsection{Horizontal lift and Sasaki metrics on $TM$}

Let $M$ be a smooth manifold, let $h$ be a non-degenerate $(0,2)$-tensor field on $M$, and let $\nabla$ be an affine connection on $M$. The horizontal lift $h^H$ of $h$ on $TM$ with respect to $\nabla$ is defined by:
$$h^H(X^H,Y^H)=0$$
$$h^H(X^V,Y^V)=0$$
$$h^H(X^H,Y^V)=h(X,Y),$$
where $X,Y \in C^{\infty}(TM)$, $X^H, Y^H$ are the horizontal lifts and $X^V,Y^V$ are the vertical lifts of $X,Y$ respectively.

Let $\pi :TM \rightarrow M$ be the canonical projection and ${\pi}_*:T(TM) \rightarrow TM$ be the tangent map of $\pi$. If $a \in TM$ and $A \in T_a(TM)$, then ${\pi}_*(A) \in T_{\pi (a)}M$ and we denote by ${\chi}_a$ the standard identification between $T_{\pi (a)}M$ and its tangent space $ T_a(T_{\pi (a)}M)$.

Let ${\Psi}^{\nabla} :TM \oplus T^{*}M \rightarrow T(TM)$ be the bundle morphism defined by:
\begin{equation} \label{m40} {\Psi}^{\nabla}(X+\alpha):=X^H_a+{\chi}_a(h^{-1}(\alpha)),\end{equation}
where $a \in TM$ and $X^H_a$ is the horizontal lift of $X\in T_{\pi (a)}M$.

Let $\left\{ x^{1},...,x^{n}\right\} $ be local coordinates on $M$, let $\left\{ {\tilde{x}}^{1},..., {\tilde{x}}^{n},y^1,...,y^n\right\} $ be respectively the corresponding local coordinates on $TM$ and let $\{X_1,...,X_n, \dfrac{\partial }{\partial
{y^{1}}},.., \dfrac{\partial }{\partial
{y^{n}}}\}$ be a local frame on $T(TM)$, where $X_i=\dfrac{\partial }{\partial
{\tilde{x}}^{i}}$. The horizontal lift of $\dfrac{\partial }{\partial
{{x}}^{i}}$ is defined by:
$$(\dfrac{\partial }{\partial
{{x}}^{i}})^H:=X_i-y^k{\Gamma}^l_{ik}{\dfrac{\partial }{\partial
y^{l}}}$$
and we will denote $X_i^H=:(\dfrac{\partial }{\partial
{{x}}^{i}})^H$. Moreover, the vertical lift of $\dfrac{\partial }{\partial
{{x}}^{i}}$ is defined by:
$$(\dfrac{\partial }{\partial
{{x}}^{i}})^V:=\dfrac{\partial }{\partial
{{y}}^{i}},$$
where $i,k,l$ run from $1$ to $n$ and $\Gamma^k_{il}$ are the Christoffel's symbols of $\nabla$.

Let ${\Psi}^{\nabla} :TM \oplus T^{*}M \rightarrow T(TM)$ be the bundle morphism defined before (which is an isomorphism on the fibres). In local coordinates, we have the following expressions:
$${\Psi}^{\nabla}\left(\dfrac{\partial }{\partial
{x^{i}}}\right)=X_i^H$$
$${\Psi}^{\nabla}\left({dx^j}\right)=h^{jk} \dfrac{\partial}{\partial y^k}.$$

We denote by $\bar h$ the pull-back tensor of $h^H$ on $TM \oplus T^*M$:
$$\bar h (\sigma,\tau):=({{\Psi}^{\nabla}})^*(h^H)(\sigma,\tau):= h^H ({\Psi}^{\nabla}(\sigma),{\Psi}^{\nabla}(\tau)),$$
for any $\sigma, \tau \in C^{\infty}(TM \oplus T^*M)$. Remark that
$\bar h$ is related to the indefinite metric or to the symplectic structure of $TM \oplus T^*M$ as follows.

\begin{proposition}
If $h$ is a symmetric tensor, then:
$$\bar h=-2<\cdot , \cdot >.$$
If $h$ is a skew-symmetric tensor, then:
$$\bar h=-2(\cdot , \cdot ).$$
\end{proposition}
\begin{proof}
Let $\sigma=X+\alpha$, $\tau=Y+\beta$, $X,Y \in C^{\infty}(TM)$, $\alpha, \beta \in C^{\infty}(T^*M)$. Then:
$$\bar h(\sigma,\tau)=h^H(X^H+{\Psi}^{\nabla}(\alpha),Y^H+{\Psi}^{\nabla}(\beta))=$$
$$=h^H({\Psi}^{\nabla}(\alpha),Y^H)+h^H(X^H,{\Psi}^{\nabla}(\beta)).$$

In local coordinates, let $X=X^i \dfrac{\partial }{\partial
{x^{i}}}$, $Y=Y^j \dfrac{\partial }{\partial
{x^{j}}}$, $\alpha={\alpha}_k dx^k$, $\beta={\beta}_ldx^l$ and we get:
$$\bar h(\sigma,\tau)=h^H({\alpha}_k h^{kr} \dfrac{\partial }{\partial
{y^{r}}},Y^j X_j^H)+h^H(X^iX_i^H,{\beta}_lh^{ls} \dfrac{\partial }{\partial
{y^{s}}})=$$
$$={\alpha}_kY^jh^{kr}h_{rj}+X^i{\beta}_lh^{ls}h_{is}={\alpha}_kY^j {\delta}^k_j \pm X^i{\beta}_l {\delta}^l_i=$$
$$=\alpha (Y) \pm \beta (X),$$
where we denoted by $\delta$ the Kronecker's symbol and the sign $+$ is for $h$ symmetric, $-$ is for $h$ skew-symmetric. Then we get the statement.
\end{proof}

The Sasaki $(0,2)$-tensor field $h^S$ on $TM$, with respect to $\nabla$, is naturally defined by:
$$h^S(X^H,Y^H)=h(X,Y)$$
$$h^S(X^V,Y^V)=h(X,Y)$$
$$h^S(X^H,Y^V)=0,$$
where $X,Y \in C^{\infty}(TM)$, $X^H, Y^H$ are the horizontal lifts and $X^V,Y^V$ are the vertical lifts of $X,Y$ respectively.

We denote by $\bar{h}^S$ the pull-back tensor of $h^S$ on $TM \oplus T^*M$:
$$\bar{h}^S (\sigma,\tau):=({{\Psi}^{\nabla}})^*(h^S)(\sigma,\tau):= h^S ({\Psi}^{\nabla}(\sigma),{\Psi}^{\nabla}(\tau)),$$
for any $\sigma, \tau \in C^{\infty}(TM \oplus T^*M).$

\begin{proposition} If $h$ is a non-degenerate $(0,2)$-tensor field on $M$, then:
$$\bar{h}^S (X+\alpha,Y+\beta)=\check h(X+\alpha,Y+\beta)=h(X,Y)+h(h^{-1}(\alpha),h^{-1}(\beta)),$$
for any $X,Y \in C^{\infty}(TM)$ and $\alpha, \beta \in C^{\infty}(T^*M).$
\end{proposition}
\begin{proof} Let $\sigma=X+\alpha$, $\tau=Y+\beta$, $X,Y \in C^{\infty}(TM)$, $\alpha, \beta \in C^{\infty}(T^*M)$. Then:
$$\bar{h}^S(\sigma,\tau)=h^S(X^H+{\Psi}^{\nabla}(\alpha),Y^H+{\Psi}^{\nabla}(\beta))=$$
$$=h^S(X^H,Y^H)+h^S({\Psi}^{\nabla}(\alpha),{\Psi}^{\nabla}(\beta))=$$
$$=h(X,Y)+h^S({\Psi}^{\nabla}(\alpha),{\Psi}^{\nabla}(\beta)).$$

In local coordinates, let $\alpha={\alpha}_k dx^k$, $\beta={\beta}_ldx^l$ and we get:
$$h^S({\Psi}^{\nabla}(\alpha),{\Psi}^{\nabla}(\beta))=h^S({\alpha}_k h^{kr} \dfrac{\partial }{\partial
{y^{r}}},{\beta}_lh^{ls} \dfrac{\partial }{\partial
{y^{s}}})=$$
$$={\alpha}_k{\beta}_l h^{kr}h^{ls}h_{rs}=h(h^{-1}(\alpha),h^{-1}(\beta)).$$
Then we get the statement.
\end{proof}

\subsection{Quasi-statistical structures on cotangent bundles}

Given an affine connection on $M$, the splitting in horizontal and vertical subbundles identifies $T(T^*M)$ with the pull-back bundle ${\pi}^*(TM \oplus T^*M)$, where $\pi :T^*M \rightarrow M$ is the canonical projection map. In particular, given a connection on $TM \oplus T^*M$, we can define the pull-back connection on ${\pi}^*(TM \oplus T^*M)$.

A direct computation gives the following:

\begin{proposition} The pull-back connection $\tilde \nabla$ of $\hat \nabla$ on $T^*M$ is defined, in local coordinates,  by:
$${\tilde \nabla}_{(\dfrac{\partial }{\partial
{{x}}^{i}})^H} (\dfrac{\partial }{\partial
{{x}}^{j}})^H={\Gamma}^k_{ij} (\dfrac{\partial }{\partial
{{x}}^{k}})^H$$
$${\tilde \nabla}_{(\dfrac{\partial }{\partial
{{x}}^{i}})^H} \dfrac{\partial }{\partial
{{y}}_{j}}=(\dfrac{\partial h^{jk}}{\partial
{{x}}^{i}}+h^{jl}{\Gamma}^k_{il})h_{rk} \dfrac{\partial }{\partial
{{y}}_{r}}$$
$${\tilde \nabla}_{\dfrac{\partial }{\partial
{{y}}_{j}}} (\dfrac{\partial }{\partial
{{x}}^{i}})^H=0$$
$${\tilde \nabla}_{\dfrac{\partial }{\partial
{{y}}_{i}}}\dfrac{\partial }{\partial
{{y}}_{j}}=0.$$
\end{proposition}

In local coordinates, the torsion $T^{\tilde \nabla}$ of $\tilde \nabla$ is:
$$T^{\tilde \nabla}({(\dfrac{\partial }{\partial
{{x}}^{i}})^H},{(\dfrac{\partial }{\partial
{{x}}^{j}})^H})=({\Gamma}^k_{ij}-{\Gamma}^k_{ji}) (\dfrac{\partial }{\partial
{{x}}^{k}})^H-y_l R^l_{ijk}{\dfrac{\partial }{\partial
{{y}}_{k}}}  $$
$$T^{\tilde \nabla}({\dfrac{\partial }{\partial
{{y}}_{i}}},{(\dfrac{\partial }{\partial
{{x}}^{j}})^H})=-((\dfrac{\partial h^{ik}}{\partial
{{x}}^{j}}+h^{il}{\Gamma}^k_{jl})h_{rk}+{\Gamma}^i_{jk}){\dfrac{\partial }{\partial
{{y}}_{r}}}  $$
$$T^{\tilde \nabla}({\dfrac{\partial }{\partial
{{y}}_{i}}},{\dfrac{\partial }{\partial
{{y}}_{j}}})=0$$
and the curvature $R^{{\tilde \nabla}}$ of ${\tilde \nabla}$, which is the pull-back of $R^{{\hat \nabla}}$, is:
$$R^{{\tilde \nabla}}({\dfrac{\partial }{\partial
{{y}}_{i}}},{\dfrac{\partial }{\partial
{{y}}_{j}}})=0$$
$$R^{\tilde \nabla}(({\dfrac{\partial }{\partial
{{x}}^{i}}})^H,{\dfrac{\partial }{\partial
{{y}}_{j}}})=0$$
$$R^{\tilde \nabla}(({\dfrac{\partial }{\partial
{{x}}^{i}}})^H,({\dfrac{\partial }{\partial
{{x}}^{j}}})^H){\dfrac{\partial }{\partial
{{y}}_{k}}}=h^{kr}R^l_{ijr}h_{ls}{\dfrac{\partial }{\partial
{{y}}_{s}}}$$
$$R^{\tilde \nabla}(({\dfrac{\partial }{\partial
{{x}}^{i}}})^H,({\dfrac{\partial }{\partial
{{x}}^{j}}})^H)({\dfrac{\partial }{\partial
{{x}}^{k}}})^H=(R^{\nabla}({\dfrac{\partial }{\partial
{{x}}^{i}}},{\dfrac{\partial }{\partial
{{x}}^{j}}}){\dfrac{\partial }{\partial
{{x}}^{k}}})^H.$$

Therefore we get:

\begin{proposition}
$\nabla$ is flat if and only if ${ \tilde \nabla}$ is flat.
\end{proposition}

\begin{theorem}
Let $(M,h,\nabla)$ be a quasi-statistical manifold such that $\nabla$ is flat. Then $(T^*M, h^{S^*}, \tilde \nabla)$ is a flat quasi-statistical manifold.
\end{theorem}
\begin{proof} Let us compute $d^{{\tilde \nabla}}h^{S^*}$. From the definition of $h^{S^*}$ and ${{\tilde \nabla}}$ we get immediately:
$$(d^{{\tilde \nabla}}h^{S^*})({\dfrac{\partial }{\partial
{{y}}_{i}}},{\dfrac{\partial }{\partial
{{y}}_{j}}})=0$$
$$(d^{{\tilde \nabla}}h^{S^*})(({\dfrac{\partial }{\partial
{{x}}^{i}}})^H,{\dfrac{\partial }{\partial
{{y}}_{j}}})=0$$
$$(d^{{\tilde \nabla}}h^{S^*})(({\dfrac{\partial }{\partial
{{x}}^{i}}})^H,({\dfrac{\partial }{\partial
{{x}}^{j}}})^H)({\dfrac{\partial }{\partial
{{y}}_{k}}})=-y_l R^l_{ijr}h^{kr}  $$
$$(d^{{\tilde \nabla}}h^{S^*})(({\dfrac{\partial }{\partial
{{x}}^{i}}})^H,({\dfrac{\partial }{\partial
{{x}}^{j}}})^H)({\dfrac{\partial }{\partial
{{x}}^{k}}})^H=(d^{\nabla}h)({\dfrac{\partial }{\partial
{{x}}^{i}}},{\dfrac{\partial }{\partial
{{x}}^{j}}})({\dfrac{\partial }{\partial
{{x}}^{k}}}).$$
Then we get the statement.
\end{proof}

Moreover, considering the Patterson-Walker metric, $\tilde h_{\pm}$, we get the following:

\begin{theorem}
Let $(M,h,\nabla)$ be a quasi-statistical manifold such that $\nabla$ is flat. Then $(T^*M, \tilde h_{\pm}, \tilde \nabla)$ is a quasi-statistical manifold.
\end{theorem}
\begin{proof} Let us compute $d^{{\tilde \nabla}}\tilde h_{\pm}$. From the definition of $\tilde h_{\pm}$ and ${{\tilde \nabla}}$ we get immediately:
$$(d^{{\tilde \nabla}}\tilde h_{\pm})({\dfrac{\partial }{\partial
{{y}}_{i}}},{\dfrac{\partial }{\partial
{{y}}_{j}}})=0$$
$$(d^{{\tilde \nabla}}\tilde h_{\pm})(({\dfrac{\partial }{\partial
{{x}}^{i}}})^H,{\dfrac{\partial }{\partial
{{y}}_{j}}})=0$$
$$(d^{{\tilde \nabla}}\tilde h_{\pm})(({\dfrac{\partial }{\partial
{{x}}^{i}}})^H,({\dfrac{\partial }{\partial
{{x}}^{j}}})^H)({\dfrac{\partial }{\partial
{{x}}^{k}}})^H=-y_l R^l_{ijk}  $$
$$(d^{{\tilde \nabla}}\tilde h_{\pm})(({\dfrac{\partial }{\partial
{{x}}^{i}}})^H,({\dfrac{\partial }{\partial
{{x}}^{j}}})^H)({\dfrac{\partial }{\partial
{{y}}_{k}}})=\pm h^{kl}(d^{\nabla}h)({\dfrac{\partial }{\partial
{{x}}^{i}}},{\dfrac{\partial }{\partial
{{x}}^{j}}})({\dfrac{\partial }{\partial
{{x}}^{l}}}),$$
where the sign $+$ is for $h$ symmetric, $-$ is for $h$ skew-symmetric. Then we get the statement.
\end{proof}

\begin{definition}
A quasi-statistical manifold $(M,h,\nabla)$ such that $\nabla$ is flat is called a \textit{ Hessian manifold}.
\end{definition}

Therefore we get:
\begin{corollary}
If $(M,h,\nabla)$ is a Hessian manifold, then $(T^*M, h^{S^*}, {\tilde\nabla})$ and $(T^*M, \tilde h_{\pm}, {\tilde\nabla})$ are Hessian manifolds.
\end{corollary}

\subsection{Quasi-statistical structures on tangent bundles}

Given a non-degenerate $(0,2)$-tensor field $h$ on $M$, we have an isomorphism between $T(T^*M)$ and $T(TM)$. The connection $\tilde{\tilde \nabla}$ on $TM$ corresponding to $\tilde \nabla$ on $T^*M$, is the following:
$$\tilde{\tilde \nabla}_{(\dfrac{\partial }{\partial
{{x}}^{i}})^H} (\dfrac{\partial }{\partial
{{x}}^{j}})^H={\Gamma}^k_{ij} (\dfrac{\partial }{\partial
{{x}}^{k}})^H$$
$$\tilde{\tilde \nabla}_{(\dfrac{\partial }{\partial
{{x}}^{i}})^H}\dfrac{\partial }{\partial
{{y}}^{j}}={\Gamma}^k_{ij}\dfrac{\partial }{\partial
{{y}}^{k}}$$
$$\tilde{\tilde \nabla}_{\dfrac{\partial }{\partial
{{y}}^{j}}} (\dfrac{\partial }{\partial
{{x}}^{i}})^H=0$$
$$\tilde{\tilde \nabla}_{\dfrac{\partial }{\partial
{{y}}^{i}}} \dfrac{\partial }{\partial
{{y}}^{j}}=0.$$

In local coordinates, the torsion $T^{\tilde{\tilde \nabla}}$ of $\tilde{\tilde \nabla}$ is:
$$T^{\tilde{\tilde \nabla}}({(\dfrac{\partial }{\partial
{{x}}^{i}})^H},{(\dfrac{\partial }{\partial
{{x}}^{j}})^H})=({\Gamma}^k_{ij}-{\Gamma}^k_{ji}) (\dfrac{\partial }{\partial
{{x}}^{k}})^H-y^l R^k_{ijl}{\dfrac{\partial }{\partial
{{y}}^{k}}}  $$
$$T^{\tilde{\tilde \nabla}}({\dfrac{\partial }{\partial
{{y}}^{i}}},{(\dfrac{\partial }{\partial
{{x}}^{j}})^H})=0 $$
$$T^{\tilde{\tilde \nabla}}({\dfrac{\partial }{\partial
{{y}}^{i}}},{\dfrac{\partial }{\partial
{{y}}^{j}}})=0 $$
and the curvature $R^{\tilde{\tilde \nabla}}$ of $\tilde{\tilde \nabla}$ is:
$$R^{\tilde{\tilde \nabla}}({\dfrac{\partial }{\partial
{{y}}^{i}}},{\dfrac{\partial }{\partial
{{y}}^{j}}})=0$$
$$R^{\tilde{\tilde \nabla}}(({\dfrac{\partial }{\partial
{{x}}^{i}}})^H,{\dfrac{\partial }{\partial
{{y}}^{j}}})=0$$
$$R^{\tilde{\tilde \nabla}}(({\dfrac{\partial }{\partial
{{x}}^{i}}})^H,({\dfrac{\partial }{\partial
{{x}}^{j}}})^H){\dfrac{\partial }{\partial
{{y}}^{k}}}=R^{l}_{ijk}{\dfrac{\partial }{\partial
{{y}}^{l}}}$$
$$R^{\tilde{\tilde \nabla}}(({\dfrac{\partial }{\partial
{{x}}^{i}}})^H,({\dfrac{\partial }{\partial
{{x}}^{j}}})^H)({\dfrac{\partial }{\partial
{{x}}^{k}}})^H=R^{\nabla}({\dfrac{\partial }{\partial
{{x}}^{i}}},{\dfrac{\partial }{\partial
{{x}}^{j}}}){\dfrac{\partial }{\partial
{{x}}^{k}}}.$$

Therefore we get:

\begin{proposition}
$\nabla$ is flat if and only if $\tilde { \tilde \nabla}$ is flat.
\end{proposition}

\begin{theorem}
Let $(M,h,\nabla)$ be a quasi-statistical manifold such that $\nabla$ is flat. Then $(TM, h^S, \tilde {\tilde\nabla})$ is a flat quasi-statistical manifold.
\end{theorem}
\begin{proof} Let us compute $d^{\tilde{\tilde \nabla}}h^S$. From the definition of $h^S$ and ${\tilde{\tilde \nabla}}$ we get immediately:
$$(d^{\tilde{\tilde \nabla}}h^S)({\dfrac{\partial }{\partial
{{y}}^{i}}},{\dfrac{\partial }{\partial
{{y}}^{j}}})=0$$
$$(d^{\tilde{\tilde \nabla}}h^S)(({\dfrac{\partial }{\partial
{{x}}^{i}}})^H,{\dfrac{\partial }{\partial
{{y}}^{j}}})=0$$
$$(d^{\tilde{\tilde \nabla}}h^S)(({\dfrac{\partial }{\partial
{{x}}^{i}}})^H,({\dfrac{\partial }{\partial
{{x}}^{j}}})^H)({\dfrac{\partial }{\partial
{{y}}^{k}}})=-y^l R^r_{ijl}h_{rk}  $$
$$(d^{\tilde{\tilde \nabla}}h^S)(({\dfrac{\partial }{\partial
{{x}}^{i}}})^H,({\dfrac{\partial }{\partial
{{x}}^{j}}})^H)({\dfrac{\partial }{\partial
{{x}}^{k}}})^H=(d^{\nabla}h)({\dfrac{\partial }{\partial
{{x}}^{i}}},{\dfrac{\partial }{\partial
{{x}}^{j}}})({\dfrac{\partial }{\partial
{{x}}^{k}}}).$$
Then we get the statement.
\end{proof}

Moreover, considering the horizontal lift metric, $h^H$, we get the following:
\begin{theorem}
Let $(M,h,\nabla)$ be a quasi-statistical manifold such that $\nabla$ is flat. Then $(TM, h^H, \tilde {\tilde\nabla})$ is a quasi-statistical manifold if and only if $\nabla h=0$.
\end{theorem}
\begin{proof} Let us compute $d^{\tilde{\tilde \nabla}}h^H$. From the definition of $h^H$ and ${\tilde{\tilde \nabla}}$ we get immediately:
$$(d^{\tilde{\tilde \nabla}}h^H)({\dfrac{\partial }{\partial
{{y}}^{i}}},{\dfrac{\partial }{\partial
{{y}}^{j}}})=0$$
$$(d^{\tilde{\tilde \nabla}}h^H)(({\dfrac{\partial }{\partial
{{x}}^{i}}})^H,{\dfrac{\partial }{\partial
{{y}}^{j}}})=\pm ({\nabla}_{\dfrac{\partial }{\partial
{{x}}^{i}}} h){\dfrac{\partial }{\partial
{{x}}^{j}}}$$
$$(d^{\tilde{\tilde \nabla}}h^H)(({\dfrac{\partial }{\partial
{{x}}^{i}}})^H,({\dfrac{\partial }{\partial
{{x}}^{j}}})^H)({\dfrac{\partial }{\partial
{{x}}^{k}}})^H=-y^l R^s_{ijl}h_{sk} $$
$$(d^{\tilde{\tilde \nabla}}h^H)(({\dfrac{\partial }{\partial
{{x}}^{i}}})^H,({\dfrac{\partial }{\partial
{{x}}^{j}}})^H) ({\dfrac{\partial }{\partial
{{y}^{k}}}})=(d^{\nabla}h)({\dfrac{\partial }{\partial
{{x}}^{i}}},{\dfrac{\partial }{\partial
{{x}}^{j}}})({\dfrac{\partial }{\partial
{{x}}^{k}}}),$$
where the sign $+$ is for $h$ symmetric, $-$ is for $h$ skew-symmetric. Then we get the statement.
\end{proof}

Therefore we get:

\begin{corollary}
If $(M,h,\nabla)$ is a Hessian manifold, then $(TM, h^S, \tilde {\tilde\nabla})$ is a Hessian manifold. Moreover, if $\nabla h =0$, then $(TM, h^H, \tilde {\tilde\nabla})$ is a Hessian manifold.
\end{corollary}

\section{Norden and Para-Norden structures on cotangent and tangent bundles}

Norden manifolds, also called almost complex manifolds with B-metric, were introduced in \cite{no}. They have applications in mathematics and in theoretical physics.

\begin{definition}
A \textit{Norden manifold}, $(M,J,h)$, is an almost complex manifold $(M,J)$ with a pseudo-Riemannian metric, $h$ (called Norden metric), such that $J$ is $h$-symmetric.

Moreover, if $J$ is integrable, then $(M,J,h)$ is called \textit{complex Norden manifold}.
\end{definition}

\begin{definition} An \textit{almost Para-complex Norden manifold} (or simply, \textit{almost Para-Norden manifold}), $(M,J,h)$, is a real even dimensional smooth manifold $M$ with a pseudo-Riemannian metric, $h$, and a $(1,1)$-tensor field, $J$, such that $J^2=I$, the two eigenbundles $T^+M$, $T^-M$, associated to the two eigenvalues $+1$, $-1$, of $J$ respectively have the same rank and $J$ is $h$-symmetric.

Moreover, if $J$ is integrable, then $(M,J,h)$ is called \textit{Para-Norden manifold}.
\end{definition}

\subsection{Norden and Para-Norden structures on cotangent bundles}

Let $(M,h)$ be a pseudo-Riemannian manifold and let ${\hat J}_c$, ${\hat J}_p$ be the generalized complex structure and the generalized product structure defined by $h$ in (\ref{m10}) and (\ref{m20}) respectively. Again we will denote $\hat{J}_\mp$ for $\hat{J}_c=:\hat{J}_-$ and $\hat{J}_p=:\hat{J}_+$.

Let $\nabla$ be an affine connection on $M$ and let ${\Phi}^{\nabla} :TM \oplus T^{*}M \rightarrow T(T^*M)$ be the bundle morphism defined by (\ref{m30}). We define:
$${\tilde{J}}^{\nabla}_\mp=:{\Phi}^{\nabla} \circ {\hat J}_{\mp} \circ ({\Phi}^{\nabla})^{-1}.$$
We have immediately that $({\tilde{J}}^{\nabla}_\mp)^2=\mp I$.

\begin{proposition}  Let $\tilde h$ be the Patterson-Walker metric on $T^*M$. Then $(T^*M,{\tilde{J}}^{\nabla}_-,\tilde h)$ is a Norden manifold and  $(T^*M,{\tilde{J}}^{\nabla}_+,\tilde h)$ is an almost Para-Norden manifold. Moreover, if $(M,h,\nabla)$ is a flat quasi-statistical manifold, then $(T^*M,{\tilde{J}}^{\nabla}_-,\tilde h)$ is a complex Norden manifold and  $(T^*M,{\tilde{J}}^{\nabla}_+,\tilde h)$ is a Para-Norden manifold.
\end{proposition}
\begin{proof} In local coordinates, we get the following:
$${\tilde{J}}^{\nabla}_\mp (X_i^H)=:h_{ik}{\dfrac{\partial }{\partial
{{y}}_{k}}}$$
$${\tilde{J}}^{\nabla}_\mp (\dfrac{\partial }{\partial
{{y}}_{j}})=:\mp h^{jk}X^H_k.$$

In particular, we have:
$$\tilde h ({\tilde{J}}^{\nabla}_\mp (X_i^H),X^H_j)=h_{ij}$$
$$\tilde h ({\tilde{J}}^{\nabla}_\mp ({\dfrac{\partial }{\partial
{{y}}_{i}}}),{\dfrac{\partial }{\partial
{{y}}_{j}}})=\mp h^{ij}$$
$$\tilde h ({\tilde{J}}^{\nabla}_\mp (X_i^H),{\dfrac{\partial }{\partial
{{y}}_{j}}})=0$$
$$\tilde h ({\tilde{J}}^{\nabla}_\mp ({\dfrac{\partial }{\partial
{{y}}_{i}}}),X^H_j)=0,$$
therefore, from the symmetry of $h$, we get the first statement.

Moreover, if we compute the Nijenhuis tensor field of ${{\tilde{J}}^{\nabla}_\mp}$, we have:
$$N_{{\tilde{J}}^{\nabla}_\mp}(X^H_i,X^H_j)=\pm (h^{kl}(d^{\nabla}h)({\dfrac{\partial }{\partial
{{x}}^{i}}},{\dfrac{\partial }{\partial
{{x}}^{j}}})({\dfrac{\partial }{\partial
{{x}}^{k}}})-y_kR^k_{ijl}){\dfrac{\partial }{\partial
{{y}}_{l}}}$$
$$N_{{\tilde{J}}^{\nabla}_\mp}(X^H_i,{\dfrac{\partial }{\partial
{{y}}_{j}}})=h^{jl}(h^{sr}y_kR^k_{ils}X^H_r \mp (d^{\nabla}h)({\dfrac{\partial }{\partial
{{x}}^{i}}},{\dfrac{\partial }{\partial
{{x}}^{l}}})({\dfrac{\partial }{\partial
{{x}}^{r}}}){\dfrac{\partial }{\partial
{{y}}_{r}}})$$
$$N_{{\tilde{J}}^{\nabla}_\mp}({\dfrac{\partial }{\partial
{{y}}_{i}}},{\dfrac{\partial }{\partial
{{y}}_{j}}})=h^{ik}h^{jl}(h^{ps}(d^{\nabla}h)({\dfrac{\partial }{\partial
{{x}}^{l}}},{\dfrac{\partial }{\partial
{{x}}^{k}}})({\dfrac{\partial }{\partial
{{x}}^{p}}})X^H_s+y_sR^s_{klr}{\dfrac{\partial }{\partial
{{y}}^{r}}}).$$
Then the proof is complete.
\end{proof}

\begin{remark} If $h$ is a non-degenerate skew-symmetric $(0,2)$-tensor field on $M$, then the same construction gives rise to a Hermitian, respectively Para-Hermitian, structure on $T^*M$.
\end{remark}

\subsection{Norden and Para-Norden structures on tangent bundles}

Let $(M,h)$ be a pseudo-Riemannian manifold and let ${\hat J}_c$, ${\hat J}_p$ be the generalized complex structure and the generalized product structure defined by $h$ in (\ref{m10}) and (\ref{m20}) respectively. Again we will denote $\hat{J}_\mp$ for $\hat{J}_c=:\hat{J}_-$ and $\hat{J}_p=:\hat{J}_+$.

Let $\nabla$ be an affine connection on $M$ and let ${\Psi}^{\nabla} :TM \oplus T^{*}M \rightarrow T(TM)$ be the bundle morphism defined by (\ref{m40}). We define:
$${\bar J}^{\nabla}_\mp=:{\Psi}^{\nabla} \circ {\hat J}_{\mp} \circ ({\Psi}^{\nabla})^{-1}.$$

Let $X \in C^{\infty}(TM)$ and let  $X^H$, $X^V$ be respectively the horizontal and vertical lift of $X$. We have immediately that
$${\bar J}^{\nabla}_\mp(X^H)=X^V$$
$${\bar J}^{\nabla}_\mp(X^V)=\mp X^H.$$

A direct computation gives the following:
\begin{proposition}  Let $h^H$ be the horizontal lift metric of $h$ on $TM$. Then $(TM,{\bar J}^{\nabla}_-, h^H)$ is a Norden manifold and  $(TM,{\bar J}^{\nabla}_+, h^H)$ is an almost Para-Norden manifold.
\end{proposition}

\begin{remark} The almost complex structure ${\bar J}^{\nabla}_-$ is the canonical almost complex structure of $TM$ defined in \cite{d}. In particular, it is integrable if and only if $\nabla$ is flat and torsion-free.
\end{remark}

\textit{Adara M. Blaga}

\textit{Department of Mathematics}

\textit{West University of Timi\c{s}oara}

\textit{Bld. V. P\^{a}rvan nr. 4, 300223, Timi\c{s}oara, Rom\^{a}nia}

\textit{adarablaga@yahoo.com}

\bigskip

\textit{Antonella Nannicini}

\textit{Department of Mathematics and Informatics "U. Dini"}

\textit{University of Florence}

\textit{Viale Morgagni, 67/a, 50134, Firenze, Italy}

\textit{antonella.nannicini@unifi.it}

\end{document}